\documentclass[12pt]{amsart}
\usepackage{amsmath, amsfonts, amsthm, amssymb}
\usepackage[mathscr]{eucal}
\usepackage[colorlinks=true, linkcolor=blue]{hyperref}
\usepackage{tikz-cd}

\newcommand{\Af}{\mathbb{A}_{f}}

\newcommand{\C}{\mathbb{C}}
\newcommand{\F}{\mathbb{F}}
\newcommand{\Fp}{\mathbb{F}_p}
\newcommand{\Fpx}{\mathbb{F}_p^\times}

\newcommand{\N}{\mathbb{N}}

\newcommand{\Q}{\mathbb{Q}}
\newcommand{\Ql}{\Q_l}
\newcommand{\Qp}{\Q_p}
\newcommand{\R}{\mathbb{R}}

\newcommand{\Z}{\mathbb{Z}}

\newcommand{\Zlt}{\Z/l^t}
\newcommand{\Zl}{{\Z_l}}
\newcommand{\Zn}{\Z/n}

\newcommand{\cG}{\mathcal{G}}

\newcommand{\cO}{\mathcal{O}}

\newcommand{\cT}{\mathcal{T}}

\newcommand{\scC}{\mathscr{C}}

\newcommand{\hGQ}{{\widehat{G(\Q)}}}

\newcommand{\cCKZn}{\scC(K,\Zn)}
\newcommand{\cCLZlt}{\scC(L,\Zlt)}
\newcommand{\cCLZl}{\scC(L,\Zl)}
\newcommand{\cCLZn}{\scC(L,\Zn)}

\newcommand{\GammaK}{\Gamma(K)}
\newcommand{\GammaL}{{\Gamma(L)}}

\def\gg{\mathfrak{g}}

\newcommand{\gk}{\mathfrak{k}}

\newcommand{\gp}{\mathfrak{p}}

\newcommand{\gs}{\mathfrak{s}}
\newcommand{\gt}{\mathfrak{t}}

\newcommand{\gz}{\mathfrak{z}}

\newcommand{\Cong}{\mathrm{Cong}}
\newcommand{\Corest}{\mathrm{CoRest}}

\newcommand{\cts}{\mathrm{cts}}

\newcommand{\GL}{\mathrm{GL}}

\newcommand{\Hom}{\mathrm{Hom}}

\newcommand{\ind}{\mathrm{ind}}

\newcommand{\Ktame}{K_\mathrm{tame}}
\newcommand{\Lie}{\mathrm{Lie}}
\newcommand{\limd}[1]{
	\begin{array}{c}
		\lim\\
		\stackrel{\textstyle \rightarrow}{\scriptstyle #1}
	\end{array}
}
\newcommand{\limp}[1]{
	\begin{array}{c}
		\lim\\
		\stackrel{\textstyle \leftarrow}{\scriptstyle #1}
	\end{array}
}
\newcommand{\limpd}[1]{
	\left(\begin{array}{c}
		\lim\\
		\stackrel{\textstyle \leftarrow}{\scriptstyle #1}
	\end{array}\right)^{1}
}

\newcommand{\met}{\mathrm{met}}

\newcommand{\pr}{\mathrm{pr}}
\newcommand{\rank}{\mathrm{rank}}

\newcommand{\Rest}{\mathrm{Rest}}

\newcommand{\SL}{\mathrm{SL}}

\newcommand{\Spin}{\mathrm{Spin}}
\newcommand{\Sp}{\mathrm{Sp}}

\newcommand{\Sym}{\mathrm{Sym}}

\newcommand{\tame}{\mathrm{tame}}

\newtheorem{theorem}{Theorem}

\newtheorem{lemma}{Lemma}
\newtheorem{corollary}{Corollary}
\newtheorem{proposition}{Proposition}
\newtheorem{conjecture}{Conjecture}
\newtheorem*{cor}{Corollary}
\newtheorem*{thm}{Theorem}

\theoremstyle{definition}

\theoremstyle{remark}
\newtheorem*{remark}{Remark}

\begin{document}

\title{Non-Residually finite extensions of arithmetic groups}

\author{Richard M. Hill}

\date{July 2018}

\begin{abstract}
	The aim of the article is to show that there are many finite extensions of arithmetic groups which are not residually finite.
	Suppose $G$ is a simple algebraic group over the rational numbers satisfying both strong approximation, and the congruence subgroup problem. We show that every arithmetic subgroup of $G$ has finite extensions which are not residually finite.
	More precisely, we investigate the group
	\[
		\bar H^2(\Z/n) = \limd{\Gamma} H^2(\Gamma,\Z/n),
	\]
	where $\Gamma$ runs through the arithmetic subgroups of $G$.
	Elements of $\bar H^2(\Z/n)$ correspond to (equivalence classes of) central extensions of arithmetic groups by $\Z/n$;
	non-zero elements of $\bar H^2(\Z/n)$ correspond to extensions which are not residually finite.
	We prove that $\bar H^2(\Z/n)$ contains infinitely many elements of order $n$, some of which are invariant for the action of the arithmetic completion $\hGQ$ of $G(\Q)$.
	We also investigate which of these (equivalence classes of) extensions
	lift to characteristic zero, by determining
	the invariant elements in the group
	\[
		\bar H^2(\Zl) = \limp{t} \bar H^2(\Zlt).
	\]
	We show that $\bar H^2(\Zl)^\hGQ$ is isomorphic to $\Zl^c$ for some positive integer $c$.
	When $G(\R)$ has no simple components of complex type, we prove that $c=b+m$, where $b$ is the number of simple components of $G(\R)$ and $m$ is the dimension of the centre of a maximal compact subgroup of $G$.
	In all other cases, we prove upper and lower bounds on $c$; our lower bound (which we believe is the correct number) is $b+m$.
\end{abstract}

\maketitle

\tableofcontents

\section{Introduction}

An abstract group $G$ is said to be residually finite if,
for every non-trivial element $g$, there is a subgroup $H$
of finite index in the group, which does not contain $g$.
The content of this statement is not changed if we insist that
$H$ is a normal subgroup of $G$.
This is equivalent to the statement that the
canonical map from the group to its profinite completion
is injective.

Arithmetic groups are residually finite.
Indeed, if $\Gamma$ is an arithmetic group and
$1 \ne \gamma \in \Gamma$,
then there is even a congruence subgroup which does not contain $\gamma$.
On the other hand, Deligne wrote down a central extension $\tilde\Gamma$
of $\Sp_{2n}(\Z)$ ($n\ge 2$) by $\Z$, such that $\tilde\Gamma$ is not residually finite.
More precisely, the group $\tilde\Gamma$ fits into an exact sequence:
\[
	1 \to \Z \to \tilde\Gamma \to \Sp_{2n}(\Z) \to 1,
\]
and any subgroup of finite index in $\tilde\Gamma$ contains $2\Z$.

In this note, we show that a weaker version of Deligne's result holds
for a large class of arithmetic groups.

\subsection{}
We briefly recall Deligne's construction.
The Lie group $\Sp_{2n}(\R)$, is not simply connected.
In fact, its fundamental group is isomorphic to $\Z$.
We shall write $\tilde\Sp_{2n}(\R)$ for the universal cover
of $\Sp_{2n}(\R)$, so we have an exact sequence:
$$
	1 \to \Z \to \tilde\Sp_{2n}(\R) \to \Sp_{2n}(\R) \to 1.
$$
One defines $\tilde\Gamma$ to be the preimage of $\Sp_{2n}(\Z)$
 in $\tilde\Sp_{2n}(\R)$.
Note that $\tilde\Sp_{2n}(\R)$ is a Lie group, but is not the
group of real points of an algebraic group; in fact $\Sp_{2n}$ is
simply connected as an algebraic group.
Thus $\tilde\Gamma$ is not an arithmetic group.

\subsection{}
There are some cases for which Deligne's argument generalizes easily.
Suppose $G$ is an algebraic group over $\Q$, which is simple and simply connected.
As we have seen above, the group $G(\R)$ may fail to be simply
connected with the archimedean topology;
this happens whenever a maximal compact subgroup
of $G(\R)$ has infinite centre.
We shall assume that fundamental group $\pi_{1}(G(\R))$
has more than 2 elements.
We can define just as before an extension
\[
	1 \to \pi_{1}(G(\R)) \to \tilde\Gamma \to \Gamma \to 1,
\]
where $\Gamma$ is an arithmetic subgroup of $G(\Q)$.
There is also a canonical double cover $\tilde G(\R)^{\met}$ of $G(\R)$,
called the metaplectic cover:
\[
	1 \to \mu_{2} \to \tilde G(\R)^{\met} \to G(\R) \to 1,
	\qquad
	\mu_2 = \{1,-1\}.
\]
By the universal property of the universal cover, there is a canonical
map
\[
	\pi_{1}(G(\R)) \to \mu_{2}.
\]
Deligne's argument shows that if $G$ has the congruence subgroup property,
then every subgroup of finite index in $\tilde \Gamma$
contains $\ker\big(\pi_{1}(G(\R)) \to \mu_{2}\big)$.

To show that this generalization is not vacuous, we remark that
$\pi_{1}(G(\R))$ is infinite whenever there is a Shimura variety associated to $G$,
and the congruence subgroup property is known to hold
for simple, simply connected groups of rational rank at least $2$.

In this paper, we shall deal also with groups $G$, for which Deligne's construction cannot be used.
The most easily stated consequence of our results is the following.

\begin{theorem}
	\label{thm-simplified}
	Let $G$ be a simple algebraic group over $\Q$, which is algebraically simply connected, and has positive real rank.
	Assume also that $G$ and has finite congruence kernel.
	Let $\Gamma$ be an arithmetic subgroup of $G(\Q)$.
	Then there is a finite abelian group $A$ and an extension of groups
	\[
	 	1 \to A \to \tilde\Gamma \to \Gamma \to 1,
	\]
	such that $\tilde \Gamma$ is not residually finite.
\end{theorem}

\subsection{}
It is an important open question in geometric group theory whether every Gromov-hyperbolic group is residually finite (see for example \cite{gromov-2}, \cite{bestvina}, \cite{kapovich-wise}, \cite{wise}, \cite{lubotzki-manning-wilton}).
This question turns out to be related to the following conjecture of Serre \cite{serre-cong}.

\begin{conjecture}
	\label{serre conjecture}
	Let $G/\Q$ be a simple, simply connected algebraic group of real rank $1$.
	Then the congruence kernel of $G$ is infinite.
\end{conjecture}

\noindent
As a consequence of \autoref{thm-simplified}, we obtain the following.

\begin{corollary}
	If every Gromov--hyperbolic group is residually finite then \autoref{serre conjecture} is true.
\end{corollary}

\begin{proof}
	Let $\Gamma$ be an arithmetic subgroup of a Lie group with real rank $1$.
	It is known that $\Gamma$ is Gromov--hyperbolic
	(see chapter 7 of \cite{gromov}).
	Since hyperbolicity is invariant under quasi-isometry, every finite extension of $\Gamma$ is also hyperbolic, and hence by assumption residually finite.
	If the congruence kernel were finite, then the groups $\tilde \Gamma$ from \autoref{thm-simplified} would provide a counterexample to this.
\end{proof}

In fact one can show as a consequence of the results proved here
 the following slightly more precise result.

\begin{corollary}
	Assume that every Gromov--hyperbolic group is residually finite.
	If $G/\Q$ is a simple, simply connected group of real rank $1$
	then for every positive integer $n$ there is a surjective homomorphism
	from the congruence kernel of $G$ to $\Zn$.
\end{corollary}

\emph{Acknowledgement.}
I'd like to thank Lars Louder for many useful discussions.

\section{Statement of Results}

Throughout this section, we fix a simple algebraic group $G/\Q$, such that
\begin{enumerate}
	\item
	$G$ is (algebraically) simply connected;
	\item
	$G$ has positive real rank (i.e. $G(\R)$ is not compact, and arithmetic subgroups of $G$ are infinite);
	\item
	The congruence kernel of $G/\Q$ is finite
	(and hence conjecturally the real rank of $G$ is at least $2$).
\end{enumerate}
We do not assume that $G$ is absolutely simple.

We'll show that \autoref{thm-simplified}
is a consequence of the following result.

\begin{theorem}
	\label{existence}
	Let $G/\Q$ be as described above
	and let $\Gamma$ be an arithmetic subgroup of $G(\Q)$.
	For every positive integer $n$ there is a
	subgroup $\Upsilon$ of finite index in
	$\Gamma$ and a central extension
	\[
		1 \to \Zn \to \tilde\Upsilon \to \Upsilon \to 1,
	\]
	such that $\tilde\Upsilon$ is not residually finite.
	More precisely, every subgroup of finite index in $\tilde\Upsilon$
	contains the subgroup $\Zn$.
\end{theorem}

\begin{proof}[Proof of \autoref{thm-simplified}]

We'll now show that \autoref{thm-simplified} is
a consequence of \autoref{existence}.
Let $\Gamma$ be an arithmetic group with finite congruence kernel.
By \autoref{existence}, there is a subgroup
$\Upsilon$ of finite index in $\Gamma$ and a
central extension $\tilde\Upsilon$ of $\Gamma$
by $\Zn$, such that every subgroup of finite
index in $\tilde\Upsilon$ contains $\Zn$.
Let $\sigma\in H^2(\Upsilon,\Zn)$ be the
cohomology class corresponding to this extension.
By Shapiro's lemma, there is an isomorphism
$H^2(\Upsilon,\Zn) \cong H^2(\Gamma, A)$,
where $A$ is the induced representation
$A=\ind_\Upsilon^\Gamma(\Zn)$.
We'll write $\Sigma$ for the image of $\sigma$
in $H^2(\Gamma,A)$.
Corresponding to the cohomology class $\Sigma$, there is a (non-central) extension $\tilde\Gamma$ of $\Gamma$ by $A$.
These two group extensions are related by the following
commutative diagram:
\[
	\begin{tikzcd}
		1 \rar & A \rar & \tilde\Gamma \rar{\pr} & \Gamma \rar & 1 & (\Sigma)\\
		1 \rar & \ind_\Upsilon^\Gamma(\Zn) \rar \uar[equal] \dar[two heads] & \pr^{-1}(\Upsilon) \rar \uar[hook] \dar[two heads] &		\Upsilon \rar \uar[hook] \dar[equal] & 1\\
		1 \rar & \Zn \rar& \tilde\Upsilon \rar & \Upsilon \rar & 1 & (\sigma)\\
	\end{tikzcd}
\]

Suppose for the sake of argument that $\tilde\Gamma$ is residually finite.
Hence the subgroup $\pr^{-1}\Upsilon$ is residually finite.
There is therefore a subgroup $\Phi \subset \tilde\Upsilon$ of finite index, such that $\Phi \cap A$ is trivial.
The image of $\Phi$ in $\tilde\Upsilon$ is then a subgroup of finite index in $\tilde\Upsilon$, whose intersection with $\Zn$ is trivial.
This is a contradiction.
\end{proof}

\subsection{Some refinements of \autoref{existence}}

Let $G/\Q$ be simple, simply connected, and have real rank at least $1$.
Furthermore assume that the congruence kernel of $G$ is finite (and hence, conjecturally at least, the real rank of $G$ is at least $2$).
Fix an arithmetic subgroup $\Gamma$ of $G$.


Suppose that we have a central extension of $\Gamma$ by $\Zn$ as follows:
\[
	1 \to \Zn \to \tilde\Gamma \to \Gamma \to 1.
\]
We shall write $\sigma\in H^{2}(\Gamma,\Zn)$ for the cohomology class of
this extension.
Suppose for a moment that $\tilde\Gamma$ is residually finite.
We can then find a subgroup $\Upsilon \subset \tilde\Gamma$ of finite index,
such that the intersection of $\Upsilon$ with $\Zn$ is trivial.
Hence $\Upsilon$ projects bijectively onto a subgroup of $\Gamma$,
which we shall also call $\Upsilon$.
The preimage of $\Upsilon$ in $\tilde\Gamma$ is the direct sum $\Zn\oplus\Upsilon$.
As a result of this, we know that the restriction of $\sigma$ to $\Upsilon$ is trivial.

This means that in order to construct a non-residually finite extension
of $\Gamma$, we need a non-zero element of the direct limit
\[
	\bar H^{2}(\Zn)
	=
	\limd{\Upsilon} H^{2}(\Upsilon,\Zn),
\]
where $\Upsilon$ runs over subgroups of finite index in $\Gamma$.
The argument above shows that \autoref{existence}
 is implied by the following result.

\begin{theorem}
	\label{reform1}
	For every positive integer $n$, there are infinitely many elements of order $n$ in
	$\bar H^{2}(\Zn)$.
\end{theorem}

We shall actually prove a stronger result, which needs a little more notation to state.
We shall write $\hGQ$ for the \emph{arithmetic completion} of the group $G(\Q)$, i.e.
\[
	\hGQ = \limp{\Upsilon} G(\Q)/\Upsilon,
\]
where $\Upsilon$ runs through the subgroups of finite index in $\Gamma$.
There is a natural projection $\hGQ \to G(\Af)$, and the congruence kernel is, by definition, the kernel of this map. This means that we have an extension of topological groups
\[
	1 \to \Cong(G) \to \hGQ \stackrel{\pr}{\to} G(\Af) \to 1.
\]
The group $\hGQ$ acts smoothly on $\bar H^2(\Zn)$.

Let $S$ be a finite set of prime numbers.
By an \emph{$S$-arithmetic level}, we shall mean an open subgroup $L$ of $\hGQ$ of the form
\[
	L = \pr^{-1} \left(
	\prod_{p\in S} G(\Qp) \times L^S \right),
	\qquad
	L^S
	=
	\prod_{p\not\in S} K_p,
\]
where each $K_p$ is a compact open subgroup of $G(\Qp)$, chosen so that $L$ is open in $\hGQ$.

\begin{theorem}
	\label{reform2}
	Let $L$ be an $S$-arithmetic level in $\hGQ$ for some finite set of primes $S$.
	For every positive integer $n$, there are infinitely many elements of order $n$ in
	$\bar H^{2}(\Zn)^L$.
\end{theorem}

\autoref{reform2} will be proved in section 4.
The proof requires a technical result on the cohomology of finite groups of Lie type, which is proved in section 5.
By modifying the argument slightly, one can also prove the following result.

\begin{theorem}
	Let $n$ be a positive integer.
	Then there are infinitely many elements $\sigma$ of order $n$ in $\bar H^2(\Zn)$ with the following property.
	There is a prime number $p$ depending on $\sigma$,
	such that for all primes $q\ne p$ the element $\sigma$ is fixed
	by $\pr^{-1}\left( G(\Q_q) \right)$.
\end{theorem}

\subsection{Virtual lifting to characteristic zero.}
Let $l$ be a prime number.
Any central extension of $\Gamma$ by $\Z/l^{t+1}$ gives rise to a central extension by $\Z/l^t$.
We'll say that the extension of $\Gamma$ by $\Z/l^r$ \emph{virtually lifts to characteristic zero}
if for every $t>r$ there is a arithmetic subgroup $\Upsilon_t$ of $\Gamma$ and a central extension of $\Upsilon_t$ by $\Z/l^t$, 
 such that the extensions fit into a commutative diagram of the following form.
\[
	\begin{tikzcd}
		1 \rar & \Z/l^{t} \rar \dar[two heads]
		&\tilde\Upsilon_t \rar \dar 
		&\Upsilon_t \rar \dar[hook]
		& 1 \\
		1 \rar & \Z/l^r \rar &\tilde\Gamma \rar
		& \Gamma \rar &1
	\end{tikzcd}
\]
Here the map $\Z/l^{t} \to \Z/l^r$ is the usual reduction map, and the map $\Upsilon_t \to \Gamma$ is the inclusion.

Equivalently, an element of $\bar H^2(\Z/l^r)$
virtually lifts to characteristic zero if
it is in the image of the following group.
\[
	\bar H^{2}(\Zl)
	=
	\limp{t} \bar H^{2}(\Zlt).
\]
There is a continuous action of $\hGQ$ on the cohomology group $\bar H^2(\Z_l)$.
Our next result will show that
there are indeed families of non-residually finite central extensions, which virtually lift to
characteristic zero.
Before stating the result we'll need a little notation.
The group $G \times \R$ is semi-simple over $\R$,
and decomposes as a product of finitely many simple groups $G_i/\R$.
We'll say that a simple group $G_i/\R$ is of \emph{complex type} if $G_i$ is the restriction of scalars of a group defined over $\C$, or equivalently if $G_i \times \C$ is a product of two simple groups; otherwise we say that $G_i$ is of real type.
We'll write $b_\R$ for the number of simple factors if $G\times \R$ of real type and $b_\C$ for the number of simple factors if $G\times \R$ or complex type.
We'll also write $m$ for the dimension
of the centre of a maximal compact subgroup
$K_\infty \subset G(\R)$.

\begin{theorem}
	\label{reform3}
	The group $\bar H^{2}(\Zl)^{\hGQ}$ is isomorphic to $\Zl^c$ for some positive integer $c$.
	More precisely, $c$ is in the range
	\[
		b_\R+b_\C+m
		\le
		c
		\le
		b_\R+2b_\C+m,
	\]
	where $b_\R$, $b_\C$ and $m$ as the integers defined above.
	In particular $\bar H^{2}(\Zl)$ is non-zero.
\end{theorem}

For comparison, we note that the construction of Deligne implies the bound $c \ge m$;
this is because $\pi_1(G(\R))$ has a finite index subgroup isomorphic to $\Z^m$.

As an easy consequence, of the theorem, we obtain the following:

\begin{corollary}
	Let $G/\Q$ be simple and simply connected
	with finite congruence kernel.
	There is a subgroup of $\bar H^2(\Zlt)^\hGQ$ isomorphic to $(\Zlt)^{c}$,
	all of whose elements virtually lift to characteristic zero,
	where $c$ is the positive integer in \autoref{reform3}.
\end{corollary}

\autoref{reform3} and its corollary will be proved in section 6.
The proof requires a result on the cohomology of
compact symmetric spaces, which is proved in
the appendix.

\begin{remark}
	We stress that \autoref{reform3} implies $\bar H^2(\Zl)^{\hGQ}$ is non-zero
	even in cases where $H^2(\Gamma,\C)=0$ for all arithmetic subgroups $\Gamma$ of $G(\Q)$.
	This happens when $G$ has large real rank and the symmetric space
	associated to $G$ has no complex structure, for example when $G=\SL_5/\Q$.
	The extensions constructed by the method of Deligne exist
	only in the case $m>0$; our result shows that
	$\bar H^2(\Zl)^\hGQ$ is non-zero even in cases where $m=0$.
\end{remark}

\begin{remark}
	The author believes that $\rank_\Zl  \left( \bar H^{2}(\Zl)^{\hGQ}\right)=b_\R+b_\C+m$.
	Proving this would amount to showing that the restriction map
	$H^3_\cts(G(\Ql),\Ql) \to H^3(G(\Q),\Ql)$ is surjective.
\end{remark}

As long as $G(\R)$ has no simple factors of complex type, \autoref{reform3} tells us precisely the rank of $\bar H^2(\Zl)^\hGQ$.
Some examples are given in Table 1.
In this table, $\Spin(r,s)$ denotes the Spin group of an arbitrary quadratic forms
over $\Q$ of signature $(r,s)$.
The congruence subgroup property for such groups was
established by Kneser \cite{kneser-spin}.

\begin{table}
\caption{}
\begin{center}
\begin{tabular}{|l l|c|c|c|}
	\hline
	$G$  & & $m $ & $b_\R$ & $c = \rank_{\Zl}\left( \bar H^{2}(\Zl)^{\hGQ}\right)$ \\
	\hline
	$\SL_{n}/\Q$ &($n \ge 3$) & 0 & 1 & 1\\[1mm]
	$\Sp_{2n}/\Q$ &($n \ge 2$) & 1 & 1 & 2\\[1mm]
	$\Spin(r,s)$ &($r \ge s \ge 3$) &0 &1& 1\\[1mm]
	$\Spin(r,2)$ &($r \ge 3$) &1&1& 2\\[1mm]
	$\Spin(2,2)$ & &2&2& 4\\[1mm]
	$\Rest^{k}_{\Q}(\SL_{n}/k)$ & ($n \ge 3$, $k$ totally real) &0& $[k:\Q]$ & $[k:\Q]$ \\[1mm]
	$\Rest^{k}_{\Q}(\SL_{2}/k)$ & ($k$ totally real, $k \ne \Q$) &$[k:\Q]$&$[k:\Q]$& $2[k:\Q]$ \\[1mm]
	$\Rest^{k}_{\Q}(\Sp_{2n}/k)$ & ($k$ totally real) &$[k:\Q]$&$[k:\Q]$& $2[k:\Q]$ \\[1mm]
	\hline
\end{tabular}
\end{center}
\label{default}
\end{table}%

The case $\SL_{2}/\Q$ and its forms of rank $0$ are not included in the table.
This is because these groups have infinite congruence kernel, and indeed for these groups
we have $\bar H^{2}(\Zn)=0$ and $\bar H^2(\Zl)=0$.

\section{Background material}

\subsection{Continuous cohomology}

We shall make use of the continuous cohomology groups $H^\bullet_\cts(G,A)$, where $G$ is a topological group and $A$ is an abelian topological group, which is a $G$-module via a continuous action $G \times A \to A$.

In all cases under consideration here, the group $G$ will be metrizable, locally compact, totally disconnected, separable and $\sigma$-compact.
The coefficient group $A$ will always be Polonais.
Under these restriction, the continuous cohomology groups defined in \cite{casselman-wigner} (based on continuous cocycles) are the same as those defined in \cite{moore I}, \cite{moore II}, \cite{moore III} based on Borel measurable cocycles.
This is proved in Theorem 1 of \cite{wigner}.

If $A$ is a continuous $H$-module for some closed subgroup $H$ of $G$, then we shall write $\ind_H^G(A)$ for the induced module,
consisting of all continuous functions $f:G \to A$ satisfying $f(hg)=h\cdot f(g)$ for all $g\in G$ and $h\in H$.
This agrees with the notation of \cite{casselman-wigner} but not \cite{moore I}, \cite{moore II}, \cite{moore III}.
The following version of Shapiro's lemma holds for these induced representations.

\begin{theorem}[Shapiro's Lemma]
	\label{shapiro}
	Let $H$ be a closed subgroup of $G$, where $G$ satisfies the conditions above.
	For any continuous $H$-module $A$,
	 there is a canonical isomorphism of topological groups:
	\[
		H^{\bullet}_{\cts}(G,\ind_{H}^{G}A)
		=
		H^{\bullet}_{\cts}(H,A).
	\]
\end{theorem}

\begin{proof}
	This follows Propositions 3 and 4 of \cite{casselman-wigner} in view of the remark following Proposition 4.
\end{proof}

We shall also make frequent use of the following.

\begin{theorem}[The Hochschild--Serre spectral sequence]
	Let $H$ be a closed normal subgroup of a group $G$, where $G$ satisfies the conditions above,
	and let $A$ be a continuous Polonais representation of $G$.
	If the groups $H^\bullet(H,A)$ are all Hausdorff,
	then there is a first quadrant spectral sequence converging to $H^\bullet(G,A)$, with $E_2$ sheet given by
	\[
		E_{2}^{r,s}
		=
		H^{r}_{\cts}(G,H^{s}_{\cts}(H,A)).
	\]
\end{theorem}

\begin{proof}
	This follows from Theorem 9 of \cite{moore III} in all cases under consideration.
\end{proof}

\subsection{The derived functor of projective limit}

By a \emph{projective system}, we shall mean a sequence of abelian groups $A_{t}$, indexed by $t\in\N$, and connected by group homomorphisms as follows:
\[
	A_{1} \leftarrow
	A_{2} \leftarrow
	A_{3} \leftarrow \cdots.
\]
We shall write $\limp{t}A_t$
for the projective limit of the system.
The functor $\limp{t}$ is left-exact from the category of projective systems of abelian groups to the category of abelian groups.
As such, it has derived functors $\left(\limp{t} \right)^{\bullet}A_{t}$.
It is known (Corollary 3.5.4 of \cite{weibel})
that the higher derived functors $\left(\limp{t} \right)^{n}A_{t}$ for $n\ge 2$ are all zero.

The projective system $(A_{t})$ is said to satisfy the
\emph{Mittag--Leffler property} if for every $t\in\N$, there is a $j\in\N$
with the property that for all $k>j$ the image of $A_{k}$ in
$A_{t}$ is equal to the image of $A_{j}$ in $A_{t}$.
For example, if the Abelian groups $A_{t}$ are all finite then the projective system has the Mittag--Leffler property.
Similarly, if the groups $A_t$ are all finite dimensional vector spaces connected by linear maps, then the projective system satisfies the Mittag--Leffler condition.

\begin{proposition}[Proposition 3.5.7 of \cite{weibel}]
	If the projective system $(A_{t})$ satisfies the Mittag--Leffler condition then
	\(
		\limpd{t} A_{t}
		=
		0
	\).
\end{proposition}

\begin{theorem}[Theorem 3.5.8 of \cite{weibel}]
	\label{limses}
	Let $\cdots \to C^{r}_{2}\to C^{r}_{1}$ be a projective system of cochain complexes of abelian groups, each indexed by $r \ge 0$.
	Assume that this projective system has the Mittag--Leffler property,
	and let $C^{r}=\limp{t} C^r_t$ be the projective limit of the complexes.
	Then we have
	\(
		H^{0}(C^{\bullet})
		=
		\limp{t}
		H^{0}(C^{\bullet}_{t})
	\).
	Furthermore, for every $r\ge 0$ there is a short exact sequence
	\[
		0 \to
		\limpd{i} 
		H^{r}(C^{\bullet}_{i})
		\to
		H^{r+1}(C^{\bullet})
		\to
		\limp{i}
		H^{r+1}(C^{\bullet}_{i})
		\to
		0.
	\]
\end{theorem}

As a simple example, we show how to express the cohomology of $G(\Q)$ in terms of
the cohomology of its $S$-arithmetic subgroups.
As before, we let $G/\Q$ be a simple, simply connected algebraic group, and $K_f=\prod_p K_p$ a compact open subgroup of $G(\Af)$.
We shall write $\Gamma$ for the arithmetic group $G(\Q) \cap K_f$.
More generally, if $S$ is a finite set of prime numbers, then we use the notation
$\Gamma^S$ for the corresponding $S$-arithmetic group, i.e.
\[
	\Gamma^S = G(\Q) \cap K^S,
	\qquad
	K^S = \left(\prod_{p\in S} G(\Qp) \right) \times K_f.
\]

\begin{proposition}
	\label{proposition projective S-arithmetic}
	For any field $\F$, we have
	\(
		H^\bullet(G(\Q), \F)
		=
		\limp{S} H^\bullet(\Gamma^S,\F).
	\)
	In the case $\F=\C$ we have
	\(
		H^\bullet(G(\Q),\C)
		=
		H^\bullet(\gg,\gk,\C),
	\)
	where $H^\bullet(\gg,\gk,\C)$ are the relative
	Lie algebra cohomology groups
	studied in \cite{borel-wallach}.
\end{proposition}

\begin{proof}
	For each $r \ge 0$ we shall write $C^r(\Gamma^S,\F)$ for the usual (inhomogeneous)
	cochain complex, consisting of all functions
	$f: \left(\Gamma^S\right)^r \to \F$.
	Since $G(\Q)$ is the union of the groups $\Gamma^S$,
	it follows that
	\[
		C^r(G(\Q),\F) = \limp{S}  C^r(\Gamma^S,\F).
	\]
	The maps in this projective system are restrictions of functions, and they are
	obviously surjective.
	Therefore the projective system satisfies the Mittag--Leffler condition.
	As a consequence, we have short exact sequences
	\[
		0 \to
		\limpd{S} 
		H^{r}(\Gamma^S,\F)
		\to
		H^{r+1}(G(\Q),\F)
		\to
		\limp{S}
		H^{r+1}(\Gamma^S,\F)
		\to
		0.
	\]
	By the theory of the Borel--Serre compactification (see \cite{borel serre}), the cohomology groups $H^{r}(\Gamma^S,\F)$
	are finite dimensional vector spaces.
	Therefore the projective system $\left(H^{r}(\Gamma^S,\F)\right)_S$ satisfies the
	Mittag--Leffler condition, so we have $\limpd{S} H^{r}(\Gamma^S,\F)=0$.
	
	In the case $\F=\C$, the theorem of \cite{blasius franke gruenewald} implies that
	$H^r(\Gamma^S,\C)=H^r(\gg,\gk,\C)$ whenever
	$S$ contains more than $r$ primes.
	Hence the projective limit (over $S$)
	is in this case $H^r(\gg,\gk,\C)$.
\end{proof}

\subsection{The congruence kernel}

Let $G/\Q$ be a simple, simply connected
 group with real rank at least 1.
By Kneser's strong approximation theorem
(see \cite{kneser-strongapprox}) the group $G(\Q)$ is dense is $G(\Af)$, where $\Af$ is the ring of finite ad\`eles of $\Q$.
It follows that there is an isomorphism of topological groups:
\[
	G(\Af)
	=
	\limp{\text{congruence subgroups } \Gamma}
	G(\Q)/\Gamma,
\]
where $\Gamma$ runs over the congruence subgroups of $G(\Q)$.
Recall that an \emph{arithmetic subgroup} of $G$ is any subgroup of $G(\Q)$, which is
commensurable with a congruence subgroup.
The \emph{arithmetic completion} $\hGQ$ is defined to be the completion of $G(\Q)$ with respect to the arithmetic subgroups of $G$, i.e.
\[
	\hGQ
	=
	\limp{\text{arithmetic subgroups } \Gamma} G(\Q)/\Gamma.
\]
There is a canonical surjective homomorphism $\hGQ \to G(\Af)$.
The \emph{congruence kernel} $\Cong(G)$ is defined to be the kernel of this map, so we have
a short exact sequence:
\[
	1 \to \Cong(G) \to \hGQ \to G(\Af) \to 1.
\]
The congruence kernel is trivial if and only if every arithmetic subgroup of
$G$ is a congruence subgroup.
If $G(\R)$ is simply connected as an analytic group,
 then the congruence kernel is never trivial, but may still be finite.
It has been conjectured by Serre \cite{serre-cong},
 that the congruence kernel is finite if and only if
each simple factor of $G$ over $\Q$ has real rank at least 2.
In the case that $\Cong(G)$ is finite, it is known that $\Cong(G)$ is contained in the
centre of $\hGQ$, and is a cyclic group.

\section{Proof of \autoref{reform2}}

In this section, we assume that the group $G/\Q$ is a simple, simply connected algebraic group with positive real rank.
We shall also assume that the congruence kernel $\Cong(G)$ is finite.
Hence, conjecturally that the real rank of $G$ is at least $2$.

\subsection{The groups $\cCLZn$}
\label{sect:5.1}

Let $L$ be an open subgroup of the arithmetic completion $\hGQ$.
We shall write $\GammaL$ for the group
$G(\Q)\cap L$.
Since $G(\Q)$ is dense in $\hGQ$, it follows that $\GammaL$ is dense in $L$.
If $L$ is compact and open then $\GammaL$ is an arithmetic group and $L$ is its profinite completion.
If $L$ is an $S$-arithmetic level, then
$\GammaL$ is an $S$-arithmetic group.

We shall write $\cCLZn$ for the group of continuous functions $f:L \to \Zn$.
We regard $\cCLZn$ as a $\GammaL\times L$-module, in which (for the sake of argument) $\GammaL$ acts by
left-translation and $L$ acts by right-translation.
We regard $\GammaL$ as a discrete topological group, and $L$ as a topological group with the subspace topology from $\hGQ$.
We do not assume that elements of $\cCLZn$ are uniformly continuous, and so the action of $L$ is not smooth unless $L$ is compact. The action is continuous, where $\cCLZn$ is equipped with the compact--open topology.

We shall also use the following notation, which was introduced earlier:
\[
	\bar H^\bullet(\Zn)
	=
	\limd{\Upsilon}
	H^\bullet(\Upsilon,\Zn),
\]
where $\Upsilon$ ranges of the arithmetic subgroups.

\begin{proposition}
	\label{reform4}
	For each open subgroup $L$ of $\hGQ$, there is a canonical isomorphism
	 of $L$-modules:
	\[
		\bar H^{\bullet}(\Zn)
		=
		H^{\bullet}(\GammaL,\cCLZn).
	\]
	The cohomology groups $\bar H^r(\Zn)= H^r(\GammaL,\cCLZn)$ are discrete (and hence  Hausdorff).
\end{proposition}

\begin{proof}
	As a first step, we'll show that the groups
	$H^\bullet(\GammaL,\cCLZn)$ do not depend on the level $L$.
	Let $K$ be an open subgroup of $L$.
	As a $\GammaL$-module, we have
	\[
		\cCLZn
		\cong_{\GammaL}
		\ind_{\GammaK}^{\GammaL} \cCKZn.
	\]
	By \nameref{shapiro}, we have an isomorphism of
	topological groups:
	\[
		H^{\bullet}(\GammaL,\cCLZn)
		=
		H^{\bullet}(\GammaK,\cCKZn).
	\]
	It's therefore sufficient to consider the case that
	the level $L$ is compact and open.
	Under this assumption, we have (as $\GammaL$-modules):
	\[
		\cCLZn
		=
		\limd{\Upsilon}
		\ind_{\Upsilon}^{\GammaL} \left( \Zn \right),
	\]
	where $\Upsilon$ ranges over the arithmetic subgroups of $\GammaL$.
	Since direct limits commute with cohomology, this implies
	\[
		H^{\bullet}(\GammaL,\cCLZn)
		=
		\limd{\Upsilon} H^{\bullet}\left(\GammaL, \ind_{\Upsilon}^{\GammaL} \Zn\right).
	\]
	Applying \nameref{shapiro}, we have
	\[
		H^{\bullet}\left(\GammaL,\cCLZn\right)
		=
		\limd{\Upsilon} H^{\bullet}\left(\Upsilon, \Zn\right).
	\]
	If we choose $L$ to be compact,
	then $\cCLZn$ is discrete, and therefore $H^\bullet(\GammaL,\cCLZn)$ is discrete.
\end{proof}

\begin{lemma}
	\label{cts-cohom}
	We have
	\(
		H^{0}_{\cts}(L, \cCLZn)
		=
		\Zn
	\)
	and $H^s_\cts(L,\cCLZn)=0$ for $s>0$.
	In particular the groups $H^{s}_{\cts}(L, \cCLZn)$ are Hausdorff.
\end{lemma}

\begin{proof}
	As a continuous $L$-module, we have
	\(
		\cCLZn
		=
		\ind_{1}^{L}
		(\Zn).
	\)
	The result follows from this using \nameref{shapiro}.
\end{proof}

\begin{proposition}
	\label{proposition spectral sequence}
	Let $L$ be any open subgroup of $\hGQ$.
	Then there is a first quadrant spectral sequence
	$E_2^{r,s} = H^{r}_{\cts}(L, \bar H^{s}(\Zn))$
	which converges to $H^{r+s}(\GammaL, \Zn)$.
\end{proposition}

\begin{proof}
	We've seen that
	$H^\bullet_\cts(L,\cCLZn)$ and
	$H^\bullet(\GammaL,\cCLZn)$
	are both Hausdorff.
	We therefore have two Hochschild--Serre spectral sequences,
	both of which converge to $H^{r+s}_{\cts}(\GammaL\times L,\cCLZn)$:
	\[
		H^{r}_{\cts}(L,H^{s}(\GammaL,\cCLZn)),
		\qquad
		H^{r}(\GammaL,H^{s}_{\cts}(L,\cCLZn)).
	\]	
	By \autoref{cts-cohom}, the second of these two spectral sequence collapses and we have
	$H^{\bullet}_{\cts}(\GammaL\times L,\cCLZn)= H^{\bullet}(\GammaL,\Zn)$.
	The result now follows from \autoref{reform4}.
\end{proof}

\subsection{Low degree terms}
We shall now describe some of the low degree terms of the spectral sequence of \autoref{proposition spectral sequence}.

\begin{lemma}
	\label{SmallHmod}
	With the notation introduced above,
	\[
		\bar H^{0}(\Zn)
		=
		\Zn,
		\qquad
		\bar H^{1}(\Zn)
		=
		0.
	\]
\end{lemma}

\begin{proof}
	For $\bar H^{0}$, note that for any arithmetic group $\Upsilon$,
	$$
		H^{0}(\Upsilon,\Z/p^{r})=\Zn.
	$$
	Furthermore the restriction maps from one of these groups to another,
	are all the identity map.
	For $\bar H^{1}$, we must show that for every element
	$\sigma\in H^{1}(\Upsilon,\Zn)$, there is an arithmetic subgroup
	$\Upsilon'\subset\Upsilon$, such that the restriction of $\sigma$ to $\Upsilon'$
	is zero.
	Any such $\sigma$ is a homomorphism $\Upsilon \to \Zn$,
	so we may simply set $\Upsilon' = \ker\sigma$.
\end{proof}

By \autoref{SmallHmod}, we know that
$E_{2}^{r,0}=H^{r}_{\cts}(L,\Zn)$ and
$E_{2}^{r,1}=0$.
Therefore the bottom left corner of the $E_2$ sheet of the spectral sequence looks like this:
\[
	\begin{tikzcd}
		\bar H^2(\Zn)^L \ar{drr}\\
		0 \ar{drr} &0 \ar{drr}&0 \\
		\Zn & H^1_\cts(L,\Zn) &H^2_\cts(L,\Zn) &H^3_\cts(L,\Zn)	
	\end{tikzcd}
\]
These groups all remain the same in the $E_3$ sheet,
where we have a map $\bar H^2(\Zn)^L \to H^3_\cts(L,\Zn)$.
\begin{equation}
	\label{equation E3 sheet}
	\begin{tikzcd}
		\bar H^2(\Zn)^L \ar{ddrrr}\\
		0  &0 &0 \\
		\Zn & H^1_\cts(L,\Zn) &H^2_\cts(L,\Zn) &H^3_\cts(L,\Zn)	
	\end{tikzcd}
\end{equation}
This map is part of the exact sequence:
\begin{equation}
	\label{low terms finite}
	H^{2}(\GammaL,\Zn)
	\to
	\left(\bar H^{2}(\Zn)\right)^{L}
	\to
	H^3_\cts(L, \Zn)
	\to
	H^{3}(\GammaL,\Zn).
\end{equation}

We now recall the theorem which we are proving:

\begin{thm}
	Let $S$ be a finite set of prime numbers and let $L$ be an $S$-arithmetic level.
	Then the group $\bar H^2(\Zn)^L$ contains infinitely many elements of order $n$.
\end{thm}

\begin{proof}
	Let $L$ be an $S$-arithmetic level.
	In this case the group $\GammaL$ is an $S$-arithmetic group.
	By the theory of the Borel--Serre compactification,
	there is a resolution of $\Z$ as a $\GammaL$-module consisting of finitely generated $\Z[\GammaL]$-modules.
	This implies that the cohomology groups
	$H^r(\GammaL,\Zn)$ are all finite.
	In view of this, the sequence in \autoref{low terms finite}
	has the form
	\[
		\text{finite} 
		\to
		\left(\bar H^{2}(\Zn)\right)^{L}
		\to
		H^3_\cts(L, \Zn)
		\to
		\text{finite}.
	\]
	To prove the theorem, it is therefore sufficient to show that $H^3_\cts(L,\Zn)$ contains infinitely many elements of order $n$.
	
	It will be useful to have the following notation.
	A prime number $p$ will be called a \emph{tame} prime
	if it satisfies all of the following conditions:
	\begin{enumerate}
		\item
		$p$ is not in the finite set $S$;
		\item
		$p$ is not a factor of $|\Cong(G)|$;
		\item
		$p$ is not a factor of $n$;
		\item
		$G$ is unramified over $\Qp$.
		\item
		The group $K_p$ is a maximal hyperspecial compact open subgroup of $G(\Qp)$
		(see \cite{tits}).
		This implies that if we let $K_p^0$ be the maximal pro-$p$ normal subgroup of $K_p$,
		then the quotient $G(\Fp)=K_p/K_p^0$
		is a product of some of the simply connected finite Lie groups described in \cite{steinberg}.
		\item
		$H^r(G(\Fp),\Q/\Z)=0$ for $r=1,2$.
		We recall from \cite{steinberg} that this condition is satisfied for all but finitely many of the groups $G(\Fp)$.
	\end{enumerate}
	We note that all but finitely many primes are tame.
	For each tame prime $p$, we shall write $K_p^*$ for a lift of $K_p$ to $\hGQ$;
	note that such a lift exists and is unique by conditions (2) and (6).
	The group $L$ contains the following subgroup
	\[
		K_\tame=\prod_{p \text{ tame}} K_p^*,
	\]
	Evidently, $\pr(K_\tame)$ is a direct summand of $\pr(L)$; since $K_\tame \cap \Cong(G)$ is trivial, it follows that $K_\tame$ is a direct summand of $L$.
	Hence by the K\"unneth formula,
	$H^3_\cts(K_\tame, \Zn)$ is a direct summand of
	$H^3_\cts(L, \Zn)$.
	It is therefore sufficient to prove that
	$H^3_\cts(K_\tame,\Zn)$ contains infinitely many
	elements of order $n$.

	Since the coefficient ring $\Zn$ is finite, we have (by Proposition 8, section 2.2 of \cite{serre-galoiscohomology}) a decomposition
	\[
		H^3_\cts(K_\tame ,\Zn) = \limd{U \text{ finite}} H^3_\cts\left( \prod_{p\in U} K_p, \Zn\right).
	\]
	By the K\"unneth formula, the group on the right contains a subgroup of the form
	\[
		\bigoplus_{p \text{ tame}} H^3_\cts(K_p, \Zn).
	\]
	By conditions (3) and (5) for tame primes $p$,
	we may identify $H^\bullet_\cts(K_p,\Zn)$ with $H^\bullet(G(\Fp),\Zn)$.
	To prove the theorem, it is therefore sufficient so show that there are infinitely many tame primes $p$, such that $H^3(G(\Fp),\Zn)$
	contains an element of order $n$.
	This follows from \autoref{finite lie}, which will be proved in the next section.
\end{proof}

\section{A lemma on the cohomology of finite Lie groups}

In this section we shall prove \autoref{finite lie}, which completes the proof of \autoref{reform2}.

Before stating the theorem, we note that if $G$ is an algebraic group over $\Q$, then we may write $G$ in the form $\cG \times_\Z \Q$, for some group scheme $\cG$ over $\Z$.
The groups $\cG(\Fp)$ depend on the $\cG$, not just on $G$.
Nevertheless if we alter the group scheme $\cG$ then only finitely many of the groups $\cG(\Fp)$ will change.
Because of this, the following statement makes sense, where we are writing $G(\Fp)$ in place of $\cG(\Fp)$ for some fixed choice of $\cG$.

\begin{theorem}
	\label{finite lie}
	Let $G/\Q$ be a simple, simply connected algebraic group.
	For every positive integer $n$ there are infinitely many prime numbers $p$, such that $H^3(G(\Fp),\Zn)$ contains an element of order $n$.
\end{theorem}

I assume this sort of result is known to experts, and many special cases are consequences of results in algebraic K-theory (for example the results of \cite{quillen} imply the case $\SL_r$).

In the proof we shall use the Cartan--Eilenberg theory of invariant cohomology classes, which we  recall now.
Let $T$ be a subgroup of a finite group $G$,
and let $A$ be a $G$-module.
We shall write $\Rest^G_T$ and $\Corest^T_G$ for the restriction and corestriction maps
 between $H^\bullet(G,A)$
and $H^\bullet(T,A)$.
A cohomology class $\sigma \in H^r(T,A)$ is called
\emph{invariant} if for every $g\in G$,
\[
	\Rest^{T}_{T \cap T^g} (\sigma)
	=
	\Rest^{T^g}_{T \cap T^g} (\sigma^g) .
\]
We'll use the following result.

\begin{proposition}[Chapter XII, Proposition 9.4 of \cite{cartaneilenberg}]
	\label{prop cartan eilenberg}
	Let $T$ be a subgroup of a finite group $G$.
	If $\sigma \in H^\bullet(T,A)$ is an invariant cohomology class.
	Then
	\[
		\Rest^G_T \left( \Corest^T_G (\sigma) \right) = [G:T] \cdot \sigma .
	\]
\end{proposition}

As a corollary to this, we note the following.
	
\begin{corollary}
	\label{corollary orders}
	Let $T$ be a subgroup of a finite group $G$.
	Let $d$ be a positive integer and $l$ a prime number, such that
	$\big|[G:T]\big|_l=\big|d\big|_l$.
	If $H^r(T,\Z)$ contains an invariant class of order $dl^t$ then $H^r(G,\Z)$ contains an element of order $l^t$.
\end{corollary}
	
\begin{proof}
	Let $\tau = \Corest^T_G(\sigma)$, where $\sigma$ is the invariant class on $T$ of order $dl^t$.
	By \autoref{prop cartan eilenberg}, the restriction of $\tau$ to $T$ has order $\frac{dl^t}{\gcd( dl^t, [G:T])}$.
	The condition on $d$ implies that the order of $\Rest^G_T(\tau)$ is a multiple of $l^t$.
	Hence the order of $\tau$ is a multiple of $l^t$, so some multiple of $\tau$ has order $l^t$.
\end{proof}

In order to apply the corollary, it will be useful to note the following.

\begin{lemma}
	\label{lemma valuations}
	Let $l$ be a prime number, and let $x$ be an integer such that $x \equiv 1 \bmod 2l$.
	Then for every integer $d$ we have
	\[
		| x^{d}-1 |_l = |d(x-1)|_l.
	\]
\end{lemma}

\begin{proof}
	We recall that the $l$-adic logarithm function
	$\log_l$ converges on the multiplicative group
	$1+2l\Z_l$.
	If $\log_l$ converges at an element $x$, then we have
	$|\log(x)|_l = | x-1 |_l$.
	Our congruence condition implies that
	$\log_l(x)$ and $\log_l(x^d)$ both converge,
	so we have
	\[
		|x^d-1|_l = |\log_l(x^d)|_l
		= |d \cdot \log_l(x) |_l
		= |d \cdot (x-1)|_l.
	\]
\end{proof}

\begin{proof}[Proof of \autoref{finite lie}]
	By the Chinese remainder theorem, it is sufficient to prove the theorem in the case $n=l^t$, where $l$ is a prime number.

We shall introduce some notation.
We fix a semi-simple model $\cG$ of $G$ over $\Z$, and let
$k$ be a number field such that $\cG$ splits over $\cO_k$.
Let $\cT$ be a maximal torus in $\cG$, defined and split over $\cO_k$.
Let $P$ be the lattice of algebraic characters
$\cT \to \GL_1/\cO_k$.
The roots of $\cG$ with respect to $\cT$ are elements of the lattice $P$.
Consider the element
\[
	Q = \sum_{\alpha\in\Phi} \alpha \otimes \alpha
	\;\in\; \Sym^2(P),
\]
where $\Phi$ is the set of roots.
If we identify elements of $P$ with a group of characters of the Lie algebra $\gt$ of $\cT$, then we may similarly identify
elements of $\Sym^2(P)$ with quadratic forms on $\gt$.
The element $Q$ corresponds to the restriction of the Killing form to $\gt$. Therefore $Q$ is non-zero.

Let $e$ be the largest positive integer,
such that $Q$ is a multiple of $e$ in the lattice
$\Sym^2 (P)$.
Also let $d_1,\cdots, d_r$ be the degrees of the basic polynomial invariants of the Weyl group of $G/k$ (where $r$ is the rank of $G / k$).
The smallest of these degrees is $d_1=2$, and the others depend on the root system
(see \cite{steinberg}).
	By extending the number field $k$ if necessary, we may assume that $k$ contains a primitive root of unity of order $d_1\cdots d_r\cdot e\cdot n$.
	By the Chebotarev density theorem, there are infinitely many prime numbers which split in $k$; we'll show that each of these prime numbers has the desired property.

	From now on we fix a prime number $p$ which splits in $k$, and we are attempting to show that $H^3(\cG(\Fp),\Zn)$ contains an element of order $n$.
	By abusing notation slightly we shall write $G(\Fp)$ for the group $\cG(\Fp)$.
	We may identify $G(\Fp)$
	with $\cG(\cO_k/\gp)$ for some prime ideal $\gp$ above $p$).
	We shall also write $T(\Fp)$ for the subgroup $\cT(\cO_k/\gp)$.

	Identifying $H^3(G(\Fp),\Zn)$ with the
	$n$-torsion in $H^4(G(\Fp),\Z)$,
	we see that it's sufficient to prove there is an
	element of order $n$ in $H^4(G(\Fp),\Z)$.

	We shall use the following formula for the order of
	the the group $G(\Fp)$ (see Theorem 25, in Chapter 9 of \cite{steinberg})
	\[
		|G(\Fp)|
		=
		p^N (p^{d_1}-1) \cdots (p^{d_r}-1).
	\]
	In this formula, $N$ is the number of positive roots; $r$ is the rank and $d_1,\ldots,d_r$ are the degrees of the fundamental invariants on the Weyl group.
	Note also that since $T$ is a split torus of rank $r$, we have
	\[
		| T(\Fp) |
		=
		(p-1)^r.
	\]
	Since $p$ splits in $k$ and $k$ contains an primitive $2l$-th root of unity (because $d_1=2$), we have $p \equiv 1 \bmod 2l$.
	Hence by \autoref{lemma valuations},
	\[
		|p^{d_i}-1 |_l = |d_i(p-1)|_l.
	\]
	We therefore have
	\[
		\big| [ G(\Fp) : T(\Fp) ] \big|_l
		=
		| d_1 \cdots d_r |_l.
	\]
	By \autoref{corollary orders}, it is sufficient to show that $H^4(T(\Fp),\Z)$ has an invariant element of order $d_1 \cdots d_r \cdot n$.
	It will actually be more convenient to find an invariant element of $H^4(T(\Fp),\Z) \otimes (\Fp^\times)^{\otimes 2}$; this is because such an element is canonical, whereas the invariant
	element of $H^4(T(\Fp),\Z)$ would depend on a
	choice of primitive root modulo $p$.

	We shall construct our invariant element of $H^4$ from elements of $H^2$.
	Since $T(\Fp)$ is a finite group, we have
	canonical isomorphisms:
	\[
		H^2(T(\Fp),\Z) \cong H^1(T(\Fp),\Q/\Z)
		\cong \Hom( T(\Fp),\Q/\Z ).
	\]
	Tensoring with $\Fpx$, we get
	\[
		H^2(T(\Fp),\Z) \otimes \Fpx
		\cong
		\Hom( T(\Fp),\Fpx )
		\cong
		P/(p-1)P,
	\]
	where as before, $P$ is the lattice of algebraic characters of $T$.
	
	Recall that the cohomology ring of the cyclic group $\Fp^\times$ is the symmetric algebra on $H^2(\Fp^\times, \Z)$.
	The group $T(\Fp)$ is a product of copies of $\Fpx$, and so by the K\"unneth formula,
	$H^\bullet(T(\Fp),\Z)$ contains as a subring
	the algebra
	$H^\bullet(\Fpx,\Z)^{\otimes r}$,
	which is isomorphic to the symmetric algebra
	on $H^2(\Fpx,\Z)^r$.
	More canonically, this subring is the
	symmetric algebra  on $H^2(T(\Fp),\Z)$.
	In particular, $H^4(T(\Fp),\Z)$ contains  $\Sym^2(H^2(T(\Fp),\Z))$ as a
	subgroup;
	this is the subgroup generated by cup products of elements of $H^2(T(\Fp),\Z)$\footnote{This is a proper subgroup if and only if $r\ge 3$.}.
	From this, we see that $H^4(T(\Fp),\Z)\otimes(\Fpx)^{\otimes 2}$
	contains as a subgroup the group
	\[
		\Sym^2\left(H^2(T(\Fp),\Z) \otimes\Fpx\right)
		\cong
		\Sym^2(P/(p-1)P)
		\cong
		\Sym^2(P)/(p-1).
	\]
	We claim that the following element
	of $H^4(T(\Fp),\Z)$ is an invariant cohomology class:
	\[
		q
		=
		\sum_{\alpha \in \Phi} \alpha \cup \alpha,
	\]
	where $\Phi$ is the root system of $G$ with respect to $T$.
	The element $q$ is evidently in the subgroup
	$\Sym^2(P)/(p-1)$.
	Equivalently, we can regard $q$ as the quadratic function
	$q:T(\Fp) \to (\Fp^\times)^{\otimes 2}$ defined by
	\[
		q(t) = \sum \alpha(t) \otimes \alpha(t).
	\]
	Here we are writing the group $(\Fpx)^{\otimes 2}$ additively.
	
	Suppose $g$ is an element of $G(\Fp)$ and
	suppose that both $t$ and $g^{-1}tg$ are in $T(\Fp)$.
	To show that $q$ is an invariant class, we must show that $q(t)=q(g^{-1}tg)$.
	Evidently we have
	\[
		q(g^{-1}tg) = \sum_{\alpha\in \Phi} \alpha(g^{-1}tg) \otimes \alpha(g^{-1}tg).
	\]
	The numbers $\alpha(t)$ are the non-zero eigenvalues in the action of $t$ on the Lie algebra $\gg \otimes \Fp$.
	These eigenvalues are the same as those of
	$g^{-1} t g$, and so the numbers $\alpha(t)$
	are the same (possibly in a different order) as the numbers $\alpha(g^{-1} t g)$.
	From this it follows that $q(t^g)=q(t)$,
	so $q$ is an invariant class in $H^4(T(\Fp),\Z)\otimes (\Fp^\times)^{\otimes 2}$.

	It remains to determine the order of $q$ in $H^4(T(\Fp),\Z)\otimes (\Fp^\times)^{\otimes 2}$,
	or equivalently the order of $q$ in the subgroup $\Sym^2(P)/(p-1)$.
	By definition, $q$ is the the reduction modulo $p-1$ of the element $Q\in \Sym^2(P)$.
	We defined $e$ to be the largest integer such that $Q$ is a multiple of $e$.
	Since we are assuming that $p\equiv 1 \bmod e$,
	the order of $q$ in $\Sym^2(P)/(p-1)$ is precisely $\frac{p-1}{e}$.

	To summarize, we have shown that $H^4(T(\Fp),\Z)$
	has an invariant element of order $\frac{p-1}{e}$.
	Therefore $H^4(G(\Fp),\Z)$ has an element of order
	$\frac{p-1}{d_1 \cdots d_r \cdot e}$.
	Since $p \equiv 1 \bmod (d_1 \cdots d_r \cdot e \cdot n)$ it follows that $H^4(G(\Fp),\Z)$ has an element of order $n$.
\end{proof}

\section{Proof of \autoref{reform3}}

\subsection{The groups $\bar H^2(\Zl)$.}

As before, we let $L$ be an open subgroup of the
arithmetic completion $\hGQ$, and we shall now fix a prime number $l$.
We introduce a new module
\[
	\cCLZl
	=
	\{ \text{continuous functions }
	f : L \to \Zl \}.
\]
Again, we regard the group $\cCLZl$
 as a $\GammaL \times L$-module.
We have
\[
	\cCLZl
	=
	\limp{t} \cCLZlt.
\]
We define, analogously to the notation $\bar H^\bullet(\Zlt)$,
\begin{equation}
	\label{equation H bar Zl defn}
	\bar H^\bullet (\Zl)
	=
	H^\bullet(\GammaL,\cCLZl ).
\end{equation}
The main focus of this section is to
determine the group $\bar H^2(\Zl)^\hGQ$.
We begin by establishing some easy properties
of the modules $\bar H^\bullet(\Zl)$.

\begin{proposition}
	\label{h2barZl properties}
	The cohomology groups $\bar H^\bullet(\Zl)$
	have the following properties:
	\begin{enumerate}
		\item
		The groups $\bar H^\bullet(\Zl)$
		do not depend on the open subgroup $L$ in
		their definition (\autoref{equation H bar Zl defn}).
		\item
		$\bar H^0(\Zl)=\Zl$,
		\item
		$\bar H^1(\Zl)=0$,
		\item
		\(
			\bar H^2(\Zl)
			=
			\limp{t} \bar H^2(\Zlt).
		\)
		\item
		The group $\bar H^2(\Zl)$ is torsion-free
		and contains no non-zero divisible elements.
		\item
		For any open subgroup $L$ of $\hGQ$ we have
		\(
			\bar H^2(\Zl)^L
			=
			\limp{t} \left( \bar H^2(\Zlt)^L \right)
		\).
	\end{enumerate}
\end{proposition}

\begin{proof}
	\begin{enumerate}
	\item
	Suppose $M$ is an open subgroup of $L$.
	Then we have an isomorphism of $\GammaL$-modules $\cCLZl = \ind_M^L \scC(M,\Zl)$.
	The result follows from this by \nameref{shapiro}.
	\item
	Since $\GammaL$ is dense in $L$, it follows that the $\GammaL$-invariant continuous functions on $L$ are  constant.
	This shows that $\bar H^0(\Zl)=\Zl$.
	\item[(3,4)]
	For each $r>0$ we have by \autoref{limses} a short exact sequence
	\[
		0 \to \left(\limp{t}\right)^1 H^{r-1}(\Gamma(L),\cCLZlt)
		\to \bar H^{r}(\Zl)
		\to \limp{t} H^{r}(\Gamma(L),\cCLZlt)
		\to 0.
	\]
	In the notation of the previous section, we have
	\[
		0 \to \left(\limp{t}\right)^1 \bar H^{r-1}(\Zlt)
		\to \bar H^{r}(\Zl)
		\to \limp{t} \bar H^{r}(\Zlt)
		\to 0.
	\]
	By \autoref{SmallHmod}, $\bar H^0(\Zlt)=\Zlt$ and $\bar H^1(\Zlt)=0$.
	Both of these projective systems consist of finite groups, so satisfy the Mittag--Leffler condition.
	Therefore $\left(\limp{t}\right)^1$ vanishes on both of them.
	As a result of this we have for $r=1,2$:
	\[
		\bar H^{r}(\Zl) = \limp{t} \bar H^{r}(\Zlt).
	\]
	In particular $\bar H^1(\Zl)=0$.
	\item[(5)]
	Consider the short exact sequence of modules:
	\[
		0 \to \cCLZl \stackrel{\times l^t}{\to} \cCLZl \to \cCLZlt \to 0.
	\]
	This gives the exact sequence in cohomology
	\[
		\cdots \to \bar H^1(\Zlt) \to \bar H^2(\Zl)
		\stackrel{\times l^t}{\to} \bar H^2(\Zl)
		\to \cdots.
	\]
	We already saw in \autoref{SmallHmod}
	that $\bar H^1(\Zlt)=0$.
	This shows that $\bar H^2(\Zl)$ is torsion-free.
	Suppose $\sigma$ is a divisible element in $\bar H^2(\Zl)$.
	Then the image of $\sigma$ in $\bar H^2(\Zlt)$ is a divisible element for each $t$.
	Since $\bar H^2(\Zlt)$ is a $\Zlt$-module,
	the image of $\sigma$ in $\bar H^2(\Zlt)$ must be zero.
	By (4) it follows that $\sigma=0$.
	\item[(6)]
	This follows because the functor
	$\limp{t}$ commutes with the functor $-^L$ of
	$L$-invariants.
	\end{enumerate}
\end{proof}

\begin{proposition}
	\label{proposition l-adic sequence}
	For any $S$-arithmetic level $L \subset \hGQ$,
	there is an exact sequence as follows:
	\[
		0 \to H^2_\cts(L,\Zl) \to
		H^2(\GammaL,\Zl) \to
		\bar H^2(\Zl)^L \to  
		H^3_\cts(L,\Zl) \to
		H^3(\GammaL,\Zl).
	\]
\end{proposition}

\begin{remark}
	It is tempting to suggest that the exact sequence of the proposition follows from a spectral sequence
	of the form $H^r_\cts(L,\bar H^s(\Zl)) \implies H^{r+s}(\GammaL,\Zl)$, which would be proved in the same way as in the finite coefficient case (\autoref{proposition spectral sequence}).
	Unfortunately this is not quite so simple.
	The problem is that the groups $\bar H^r(\Zl)$ will probably not be Hausdorff for $r \ge 3$, and so there is no off-the-shelf spectral sequence for us to use.
	Admittedly we could truncate at $\bar H^2(\Zl)$ to obtain a spectral sequence with three rows, or we could try to work with the more general spectral sequence constructed in \cite{flach}. Instead we've gone for a more elementary approach, and we prove the exact sequence of the proposition by taking the projective limit of such exact sequences in the finite coefficient cases.
\end{remark}

\begin{proof}	
	For any $t\ge 0$ the spectral sequence in \autoref{equation E3 sheet} gives rise to an exact sequence:
	\[
		0 \to H^2_\cts(L,\Zlt) \to
		H^2(\GammaL,\Zlt) \to
		\bar H^2(\Zlt)^L \to
		H^3_\cts(L,\Zlt) \to
		H^3(\GammaL,\Zlt).
	\]
	We shall write $A_t$ for the image of the map $H^2(\GammaL,\Zlt) \to \bar H^2(\Zlt)^L$
	and $B_t$ for the image the map
	$\bar H^2(\Zlt)^L \to	 H^3_\cts(L,\Zlt)$.
	We therefore have three exact sequences:
	\begin{eqnarray}
		\label{eqn exact 1}
		&0 \to  H^2_\cts(L,\Zlt) \to
		H^2(\GammaL,\Zlt) \to A_t \to 0,\\
		&0 \to A_t \to
		\bar H^2(\Zlt)^L \to B_t \to 0, \\
		&0 \to B_t \to
		H^3_\cts(L,\Zlt) \to
		H^3(\GammaL,\Zlt).
	\end{eqnarray}
	As $\GammaL$ is an $S$-arithmetic group,
	the cohomology groups $H^\bullet(\GammaL,\Zlt)$
	are all finite.
	From the \autoref{eqn exact 1}
	it follows that $H^2_\cts(L,\Zlt)$ and $A_t$ are both finite, and hence
	\[
		\limpd{t} H^2_\cts(L,\Zlt) = 0,
		\qquad
		\limpd{t} A_t = 0.
	\]
	From this it follows that we have exact sequences
	\begin{eqnarray*}
		&0 \to \limp{t} H^2_\cts(L,\Zlt) \to
		\limp{t} H^2(\GammaL,\Zlt) \to \limp{t} A_t \to 0,\\
		&0 \to \limp{t} A_t \to
		\bar H^2(\Zl)^L \to  \limp{t} B_t \to 0, \\
		&0 \to \limp{t} B_t \to
		\limp{t} H^3_\cts(L,\Zlt) \to
		\limp{t} H^3(\GammaL,\Zlt).
	\end{eqnarray*}	
	In the second of these we have used part (6) of \autoref{h2barZl properties}.
	Splicing the exact sequences together again, we obtain the following exact sequence:
	\begin{multline*}
		0 \to \limp{t} H^2_\cts(L,\Zlt) \to
		\limp{t} H^2(\GammaL,\Zlt) \to
		\bar H^2(\Zl)^L   \\
		\to \limp{t} H^3_\cts(L,\Zlt) \to
		\limp{t} H^3(\GammaL,\Zlt).
	\end{multline*}
	Recall again that as $\GammaL$ is an $S$-arithmetic group, the groups $H^\bullet(\GammaL,\Zlt)$ must be finite.
	Furthermore, the spectral sequence in \autoref{equation E3 sheet} shows that the groups $H^1_\cts(L,\Zlt)$ and $H^2_\cts(L,\Zlt)$ are also finite.
	As a result, all of these projective systems satisfy the Mittag--Leffler condition, so we have:
	\begin{align*}
		\limpd{t} H^1_\cts(L,\Zlt),  \quad
		\limpd{t} H^1(\GammaL,\Zlt),  \\
		\limpd{t} H^2_\cts(L,\Zlt),  \quad
		\limpd{t} H^2(\GammaL,\Zlt).
	\end{align*}
	As a result of this, we have
	\begin{align*}
		\limp{t} H^2_\cts(L,\Zlt)& = H^2_\cts(L,\Zl), \\
		\limp{t} H^2(\GammaL,\Zlt)& = H^2(\GammaL,\Zl), \\
		\limp{t} H^3_\cts(L,\Zlt)& = H^3_\cts(L,\Zl), \\
		\limp{t} H^3(\GammaL,\Zlt)& = H^3(\GammaL,\Zl).
	\end{align*}
	Substituting these into our previous exact sequence we get
	\[
		0 \to H^2_\cts(L,\Zl) \to
		H^2(\GammaL,\Zl) \to
		\bar H^2(\Zl)^L \to  
		H^3_\cts(L,\Zl) \to
		H^3(\GammaL,\Zl).
	\]
\end{proof}

\subsection{The groups $H^\bullet_\cts(L,\Zl)$.}
We shall next concentrate on the continuous
cohomology groups in the exact sequence of
\autoref{proposition l-adic sequence}.

From now on, we assume that $L$ is an $S$-arithmetic level for some finite set of primes $S$.
Recall that this means $L$ is the pre-image in
$\hGQ$ of an open subgroup of $G(\Af)$
 of the form
\[
	\prod_{p\in S} G(\Qp) \times \prod_{p\not\in S} K_p,
	\qquad
	\text{where $K_p$ is compact and open in $G(\Q_p)$.}
\]
It will be convenient to call a prime number $p$ a \emph{tame prime} if it satisfies all of the following conditions:
\begin{enumerate}
	\item
	$p \not\in S$.
	\item
	$p \ne l$.
	\item
	$G$ is unramified over $\Qp$.
	\item
	$K_p$ is a maximal hyperspecial compact open subgroup of $G(\Qp)$.
	This implies that if we let $K_p^0$ be the maximal pro-$p$ normal subgroup of $K_p$, then the group
	$G(\Fp)=K_p/K_p^0$ is a product of some of the simply connected finite groups of Lie type described in detail in \cite{steinberg}.
	\item
	$H^r(G(\Fp),\Q/\Z)=0$ for $r=1,2$.
	We recall from \cite{steinberg} that this condition is satisfied for all but finitely many of the groups $G(\Fp)$.
\end{enumerate}
We note that for tame primes $p$ we have
$H^\bullet_\cts(K_p,\Zlt)= H^\bullet(G(\Fp),\Zlt)$
by condition (2).

All but finitely many of the prime numbers are tame.
We shall write $W$ for the set of primes not in $S$ which are not tame.
The group $L/\Cong(G)$ decomposes in the form
\[
	L/\Cong(G) = L_S \times K_W \times \Ktame,
\]
where we are using the notation:
\[
	L_S = \prod_{p\in S} G(\Qp),
	\quad
	K_W = \prod_{p\in W} K_p,
	\quad
	\Ktame = \prod_{p \text{ tame}} K_p.
\]

\begin{lemma}
	\label{lemma tame cohomology}
	With the notation introduced above, we have
	$H^1_\cts(\Ktame, \Zl)=0$,
	$H^2_\cts(\Ktame, \Zl)=0$ and
	$H^3_\cts(\Ktame, \Zl)=0$.
\end{lemma}

\begin{proof}
	For each $r>0$ we have by \autoref{limses} a short exact sequence
	\[
		0 \to
		\limpd{t} H_\cts^{r-1}(\Ktame, \Zlt)
		\to H^{r}_\cts(\Ktame, \Zl) \to
		\limp{t} H^{r}_\cts(\Ktame, \Zlt)
		\to 0.
	\]
	Furthermore, for any $r$ we have
	\[
		H^{r}_\cts(\Ktame, \Zlt)
		=
		\limd{U} H^r_\cts\left( \prod_{p \in U} K_p, \Zlt\right),
	\]
	where $U$ runs through the finite sets of tame primes.
	For such primes $p$ we have
	$H^r_\cts(K_p,\Q_l/\Zl)=0$ for $r=1,2$.
	Hence by an obvious long exact sequence we have
	$H^r_\cts(K_p,\Zlt)=0$ for $r=1,2$.
	By the K\"unneth formula we have
	\[
		H^r_\cts\left( \prod_{p \in U} K_p, \Zlt\right) = 0
		\quad
		\text{ for $r=1,2$.}
	\]
	Hence $H^r_\cts(\Ktame,\Zlt)=0$ for $r=1,2$.
	Since the projective system $H^0_\cts(\Ktame,\Zlt)=\Zlt$ satisfies the Mittag--Leffler condition, we have
	$H^r_\cts(\Ktame,\Zl)=0$ for $r=1,2$.

	We'll now concentrate on the group $H^3_\cts(\Ktame,\Zl)$.
	By the short exact sequence above, together with the fact that $H^2_\cts(\Ktame,\Zlt)=0$,
	we have
	\[
		H^3_\cts(\Ktame,\Zl) = \limp{t} H^3_\cts(\Ktame,\Zlt).
	\]
	We also have (using the K\"unneth formula
	and the fact that $H^\bullet(K_p,\Zlt)=H^\bullet(G(\Fp),\Zlt)$):
	\begin{equation}
		\label{non-mittag-leffler}
		H^3_\cts(\Ktame, \Zl) =
		\limp{t} \left( \bigoplus_{p \text{ tame}} H^3(G(\Fp),\Zlt) \right)
		\subseteq
		\prod_{p \text{ tame}}
		\left( \limp{t} H^3(G(\Fp),\Zlt) \right).
	\end{equation}
	Consider any tame prime number $p$.
	Since $G(\Fp)$ is finite, we have $H^3(G(\Fp),\Ql)=0$, and therefore
	$H^3(G(\Fp),\Zl)=H^2(G(\Fp),\Ql/\Zl)=0$.
	From this it follows that
	\[
		\limp{t} H^3(G(\Fp),\Zlt) =0.
	\]
	By \autoref{non-mittag-leffler} we have $H^3_\cts(\Ktame, \Zl)=0$.
\end{proof}

	(It might be tempting to imagine that the result above can be extended further in a simple way.
	However, we note that the projective system in
	\autoref{non-mittag-leffler} does not satisfy the Mittag--Leffler condition, so we do not expect  $H^4_\cts(\Ktame,\Zl)$ to be finitely generated as a $\Zl$-module).

\begin{lemma}
	\label{lemma L-cohomology}
	Let $L$ be an $S$-arithmetic level in $\hGQ$.
	For $r=0,1,2,3$
	we have
	\[
		H^r_\cts(L/\Cong(G), \Zl)= H^r_\cts(L_S \times K_W,\Zl),
	\]
	\[
		H^r_\cts(L, \Ql) = H^r_\cts(G(\Ql),\Ql) = H^r_\Lie(\gg \otimes \Q_l,\Ql).
	\]
	Here $\gg$ is the Lie algebra of $G$ over $\Q$.
\end{lemma}

\begin{proof}
	Recall that we have a decomposition of the group $L/\Cong(G)$ in the form $L_S \times K_W \times \Ktame$. This gives rise to the following spectral sequence
	\[
		H^r_\cts(L_S \times K_T , H^s(\Ktame,\Zl))
		\implies
		H^{r+s}_\cts (L/\Cong(G), \Zl).
	\]
	From \autoref{lemma tame cohomology}, we see that the bottom left corner of the $E_2$-sheet is as follows:
	\[
		\begin{array}{cccc}
			0 & \cdots  \\
			0 & 0 & \cdots  \\
			0 & 0 & 0 & \cdots  \\
			H^0_\cts(L_S \times K_T , \Zl)&H^1_\cts(L_S \times K_T , \Zl)& H^2_\cts(L_S \times K_T , \Zl)& H^3_\cts(L_S \times K_T , \Zl)
		\end{array}
	\]
	This shows that $H^r_\cts(L/\Cong(G), \Zl)= H^r(L_S \times K_T,\Zl)$ for $r\le 3$.
	
	Since $\Cong(G)$ is a finite group we have
	\[
		H^s(\Cong(G),\Ql)
		=\begin{cases}
			\Ql & \text{if $s=0$,}\\
			0 & \text{if $s>0$.}
		\end{cases}
	\]
	Therefore the spectral sequence
	$H^r_\cts(L/\Cong(G),H^s(\Cong(G),\Ql))$
	collapses to the bottom row, so we have
	\[
		H^r_\cts(L,\Ql)
		=
		H^r_\cts(L/\Cong(G),\Ql).
	\]
	In particular for $r \le 3$ we have
	\[
		H^r_\cts(L,\Ql)
		=
		H^r_\cts\left(\prod_{p\in S} G(\Qp) \times \prod_{p\in W} K_p,\Ql\right).
	\]
	We'll calculate these cohomology groups above
	using the K\"unneth formula.
	
	Suppose first that $p$ is a prime in $W$, which is not equal to $l$.
	The group $K_p$ contains a normal pro-$p$ subgroup of finite index.
	From this it follows that
	\[
		H^s_\cts(K_p, \Ql)
		=
		\begin{cases}
			\Ql & s=0, \\
			0 & s > 0.
		\end{cases}
	\]

	Next, suppose that $p$ is a prime in $S$ which is not equal to $l$.
	We recall from \cite{casselman-wigner}
	that there is a spectral sequence which calculates the cohomology of $G(\Qp)$ in terms of the cohomology of its compact open subgroups.
	Let $K_{p}^0$ be a maximal pro-$p$ subgroup of $G(\Qp)$.
	The subgroup $K_{p}^0$
	is compact and open in $G(\Qp)$.
	There are finitely many maximal compact subgroups of $G(\Qp)$, which contain $K_p^0$;
	we call these subgroups $K_{1},\ldots,K_{n}$.
	In the spectral sequence, the $E_{1}$-sheet is given by
	\[
		E_{1}^{r,s}
		=
		\bigoplus_{i_{0}< \ldots < i_{r}} H^{s}_{\cts}(K_{i_{0}, \ldots,i_{s}},\Ql).
	\]
	Here we are using the notation
	\[
		K_{i_{1},\ldots,i_{s}}
		=
		K_{i_{1}} \cap \ldots \cap K_{i_{s}}.
	\]
	The map $E_{1}^{r-1,s} \to E_{1}^{r,s}$ is an alternating sum of restriction maps; in other
	words, its $(i_{0},\ldots,i_{r})$-component is equal to
	\[
		\sum_{j=0}^{r}
		(-1)^{j}
		\Rest_{K_{i_{0},\ldots,i_{r}}}^{K_{i_{0},\ldots,\hat i_{j},\ldots,i_{r}}}
		\left(\sigma_{i_{0}, \ldots,\hat i_{j}, \ldots, i_{r}}\right).
	\]
	As we are assuming here that $p \ne l$, the
	cohomology groups $H^{s}_{\cts}(K_{i_{0}, \ldots,i_{s}},\Ql)$ are zero for $s>0$.
	Therefore the spectral sequence consists of a single row in $E_1$; this row is the simplicial cochain complex of a simplex with $n$ vertices.
	As this simplex is contractable, we have
	\[
		H^s_\cts(G(\Qp), \Ql)
		=
		\begin{cases}
			\Ql & s=0, \\
			0 & s > 0.
		\end{cases}
	\]	
	From the K\"unneth formula we have for $r\le 3$:
	\[
		H^r_\cts(L,\Ql)
		=
		H^r_\cts(L_l,\Ql),
	\]
	where $L_l$ is either $K_l$ or $G(\Ql)$, depending on whether the prime $l$ is in $S$ or not.
	In either case we have $H^\bullet_\cts(L_l,\Ql)= H^\bullet(\gg\otimes \Ql,\Ql)$; this is proved in \cite{lazard} for $K_l$ and in \cite{casselman-wigner} for $G(\Ql)$.
	As a result of this we have
	$H^r_\cts(L,\Ql)=H^r(\gg\otimes \Ql,\Ql)$ for $r \le 3$.
\end{proof}

\begin{lemma}
	\label{lemma gg-cohomology}
	We have $H^0(\gg \otimes\Ql,\Ql)=\Ql$, $H^1(\gg\otimes\Ql,\Ql)=0$, $H^2(\gg\otimes\Ql,\Ql)=0$ and $H^3(\gg\otimes\Ql,\Ql)=\Ql^b$, where $b$ is the number of
	simple components of $G \times_\Q \C$.
	(Note that in the notation of the introduction we have $b=b_\R+2b_\C$).
\end{lemma}

\begin{proof}
	The dimension of the Lie algebra cohomology
	does not depend on the base field, so we may instead calculate the cohomology of $\gg \otimes \C$.
	There is a decomposition:
	\[
		\gg \otimes \C = \gg_1 \oplus \cdots \oplus \gg_b,
	\]
	where each $\gg_i$ is a complex simple Lie algebra.
	By Whitehead's first and second lemmas
	we have $H^r(\gg_i,\C)=0$ for $r=1,2$,
	and $H^0(\gg_i,\C)=\C$.
	It is well known (see for example section 1.6 of \cite{borel-wallach}) that $H^\bullet(\gg_i,\C)$
	is isomorphic to the singular cohomology
	of a compact connected Lie group with
	Lie algebra $\gg_i$.
	Hence by \autoref{H3symspace} each group $H^3(\gg_i,\C)$ is $1$-dimensional.
	The lemma follows from the K\"unneth formula.
\end{proof}

\subsection{The end of the proof}

Recall from \autoref{proposition l-adic sequence}
that we have an exact sequence:
\[
	0 \to H^2_\cts(L,\Zl)
	\to H^2(\GammaL,\Zl)
	\to \bar H^2(\Zl)^L
	\to H^3_\cts(L, \Zl)
	\to H^3(\GammaL,\Zl).
\]
Tensoring with $\Ql$ and using \autoref{lemma L-cohomology} and \autoref{lemma gg-cohomology},
we have an exact sequence
of $\Ql$-vector spaces.
We've seen in \autoref{h2barZl properties} that $\bar H^2(\Zl)$ is torsion-free. This implies that we may regard $\bar H^2(\Zl)$ as a subgroup of $\bar H^2(\Ql)$,
where we are using the notation
$\bar H^2(\Ql)=\bar H^2(\Zl) \otimes \Ql$.
It follows that
$\bar H^2(\Zl)^L \otimes \Ql = \bar H^2(\Ql)^L$.
In view of this, we have an exact sequence
\begin{equation}
	\label{eq Ql spaces}
	0
	\to H^2(\GammaL,\Ql) \to
	\bar H^2(\Ql)^L \to
	H^3(\gg\otimes \Ql,\Ql) \to
	H^3(\GammaL,\Ql).
\end{equation}
The vector spaces in
\autoref{eq Ql spaces} are all finite dimensional; this follows for $\bar H^2(\Ql)^L$, because it is between the finite dimensional spaces $H^2(\GammaL,\Ql)$
and $H^3(\gg\otimes\Ql,\Ql)$.

We shall next take the projective limit over $S$ of the sequences in \autoref{eq Ql spaces}.
Since the sequences consist of finite dimensional vector spaces, the derived functors $\limpd{S}$ all vanish, so we have the following exact sequence:
\[
	0 \to
	H^2(G(\Q),\Ql) \to
	\bar H^2(\Ql)^{\hGQ} \to
	H^3(\gg\otimes\Ql,\Ql) \to
	H^3(G(\Q),\Ql).
\]
Here we have used \autoref{proposition projective S-arithmetic}, which shows that the projective limit (over $S$) of the groups $H^r(\GammaL,\Ql)$ is $H^r(G(\Q),\Ql)$.

The dimensions of the groups $H^r(G(\Q),\Ql)$ are the same as those of $H^r(G(\Q),\C)$, and by
\autoref{proposition projective S-arithmetic}
these are the same as the the relative Lie algebra
cohomology groups $H^r(\gg,\gk,\C)$.
Here $\gk$ is the Lie algebra of a maximal compact
subgroup $K_\infty$ of $G(\R)$.

Recall from section 1.6 of \cite{borel-wallach} that the relative Lie algebra cohomology
groups are isomorphic to the singular cohomology
groups $H^r(X^*,\C)$, where $X^*$ is the compact dual of the symmetric space $X=G(\R)/K_\infty$.
We calculate the dimensions of these spaces in
the appendix.
The results (see \autoref{corollary sym space}) are:
\[
	\dim H^2(\gg,\gk,\C)=\dim(Z(K_\infty)).
\]
\[
	\dim H^3(\gg,\gk,\C) = \#\text{simple components of $G \times \R$ of complex type.}
\]
If we write $b_\C$ for the number of simple components
of $G\times \R$ of complex type, and $b_\R$ for the number of simple components of real type,
then the number of simple components of $G\times \C$
is $b_\R+2b_\C$.
Substituting these dimensions, we have an exact sequence of the form
\[
	0 \to
	\Ql^{\dim(Z(K_\infty))} \to
	\bar H^2(\Ql)^{\hGQ} \to
	\Ql^{b_\R+2b_\C} \to
	\Ql^{b_\C}.
\]
It follows that the dimension of $\bar H^2(\Ql)^\hGQ$ is between $\dim(Z(K_\infty))+b_\R+b_\C$ and
$\dim(Z(K_\infty))+b_\R+2b_\C$.

Since $\bar H^2(\Zl)$ is torsion-free, it follows that $\bar H^2(\Zl)^\hGQ$ is a torsion-free $\Zl$-module, which spans $\bar H^2(\Ql)^\hGQ$.
On the other hand, $\bar H^2(\Zl)$ has no non-zero divisible elements.
This implies that
$\bar H^2(\Zl)^\hGQ \cong \Zl^{c}$,
where $c= \dim \bar H^2(\Ql)^\hGQ$.
This finishes the proof of the \autoref{reform3}.

\begin{cor}
	There is a subgroup of $\bar H^2(\Zlt)^\hGQ$ isomorphic to $(\Zlt)^{c}$,
	all of whose elements virtually lift to characteristic zero,
	where $c=\rank_{\Zl} \left(\bar H^2(\Zl)^\hGQ\right)$.
\end{cor}

\begin{proof}
	We have a short exact sequence
	\[
		0 \to \cCLZl \stackrel{\times l^t}\to \cCLZl \to \scC(L,\Zlt) \to 0.
	\]
	This gives a long exact sequence containing
	the following terms
	\[
		\bar H^1(\Zlt) \to \bar H^2(\Zl)
		\stackrel{\times l^t}{\to} \bar H^2(\Zl)
		\to \bar H^2(\Zlt).
	\]
	We've shown that $\bar H^1(\Zlt)=0$, so we
	have
	\[
		0 \to \bar H^2(\Zl)
		\stackrel{\times l^t}{\to} \bar H^2(\Zl)
		\to \bar H^2(\Zlt).
	\]
	We'll write $A$ for the subgroup of elements in $\bar H^2(\Zlt)$ which virtually lift to characteristic zero.
	By definition, $A$ is the image of $\bar H^2(\Zl)$ in $H^2(\Zlt)$.
	We therefore have a short exact sequence
	\[
		0 \to \bar H^2(\Zl)
		\stackrel{\times l^t}{\to} \bar H^2(\Zl)
		\to A \to 0.
	\]
	Taking $\hGQ$-invariants, we have an exact sequence
	\[
		0 \to \bar H^2(\Zl)^\hGQ
		\stackrel{\times l^t}{\to}
		 \bar H^2(\Zl)^\hGQ
		\to A^\hGQ.
	\]
	The result follows because $\bar H^2(\Zl)^\hGQ$ is isomorphic to $\Zl^c$.
\end{proof}

\section*{Appendix: A result on compact symmetric spaces}

In this appendix, we shall calculate the low dimensional cohomology of the compact simple symmetric spaces.
Such spaces have the form $G/K$, where $G$ is a compact, connected, simple Lie group and $K$ is a closed, connected subgroup.
In the following result, we shall use the shorthand $H^\bullet(X)$ for the singular
cohomology on the topological space $X$ with coefficients in $\R$.

\begin{proposition}
	\label{H3symspace}
	Let $G$ be a compact, connected simple Lie group and $K$ a closed, connected subgroup.
	There are isomorphisms:
	\begin{align*}
		H^1(G/K)&=0,\\
		H^{2}(G/K)
		&=
		\gz(\gk)^{*},\\
		H^{3}(G/K)
		&=
		\begin{cases}
			\R & \hbox{if $K$ is trivial,}\\
			0 & \hbox{otherwise}.
		\end{cases}
	\end{align*}
	Here $\gz(\gk)^*$ denotes the dual space of the centre of the Lie algebra $\gk$ of $K$.
\end{proposition}

\begin{remark}
	One might expect to be able to look this result up in tables; however I
	didn't manage to find such tables, so I am including a proof. The result is
	a straightforward consequence of results in Borel's thesis \cite{borel-thesis}.
\end{remark}

\begin{proof}
We shall first review some results from \cite{borel-thesis}
on the cohomology of compact Lie groups and their classifying spaces.
Suppose that $G$ is compact connected Lie group.
We shall write $BG$ for the classifying space of $G$.
This is a space with a fibre bundle
$$
	\begin{matrix}
		G& \to& EG \\
		&& \downarrow \\
		&& BG
	\end{matrix}
$$
such that the cohomology of $EG$ is the cohomology of a point.

The real singular cohomology of $G$ is (as a ring) an exterior algebra whose generators
are cohomology classes $x_{1},\ldots,x_{n}$ in odd dimensions $s_{1}, \ldots,s_{n}$.
For each generator $x_{i}\in H^{s_{i}}(G)$, there is an element $y_{i}\in H^{s_{i+1}}(BG)$ called the transgression of $x_{i}$.
Furthermore the cohomology ring $H^{\bullet}(BG)$ is equal to the polynomial ring
$\R[y_{1},\ldots,y_{n}]$.
In particular $BG$ only has non-zero real cohomology in even dimensions.

As an example of this, let $T$ be an $n$-dimensional compact torus,
 and let $\gt$ be its Lie algebra.
Recall that the cohomology of $T$ is exactly the exterior algebra of $H^{1}(T)$.
Furthermore, we may identify $H^{1}(T)$ with the dual space of $\gt$.
As a result of this, we know that $H^{\bullet}(BT)$ is the algebra $\R[\gt]$
 of polynomial functions on $\gt$.
Furthermore, $H^{2n}(BT)$ is the space of homogeneous polynomials
on $\gt$ of degree $n$.

Suppose now that $G$ is a compact, connected Lie group and $T$ is a maximal torus
in $G$.
We shall write $W$ for the Weyl group of $G$ with respect to $T$.
We have a fibre bundle
$$
	\begin{matrix}
		G/T & \to & BT \\
		&&\downarrow\\
		&& BG
	\end{matrix}
	\qquad\qquad
	(BT = EG/T).
$$
Corresponding to this there is a spectral sequence
$$
	H^{r}(BG, H^{s}(G/T)) \Rightarrow H^{r+s}(BT).
$$
As $G$ is connected, it follows easily that $BG$ is simply connected.
This implies that $H^s(G/T)$ is a trivial bundle on $BG$, and so the spectral sequence takes the form
\[
	H^{r}(BG) \otimes H^{s}(G/T) \Rightarrow H^{r+s}(BT).
\]
In particular we have an edge map $H^{\bullet}(BG) \to H^{\bullet}(BT)$.
We'll use the following result, which describes this edge map.

\begin{proposition}[Proposition 27.1 of \cite{borel-thesis}]
	\label{prop borel 1}
	Let $G$ be a compact, connected Lie group
	and $T$ a maximal torus in $G$.
	The edge map $H^{\bullet}(BG) \to H^{\bullet}(BT)$ is
	injective.
	Its image is the subspace $H^{\bullet}(BT)^{W}$ of $W$-invariant
	polynomial functions on $\gt$.
\end{proposition}

As a result of this proposition, we know that for semi-simple $G$ we have
$H^{2}(BG) =0$.
This is because there are no $W$-invariant linear forms on $\gt$.
For simple $G$ we have
\[
	H^{4}(BG) = \R.
\]
Recall that $H^{4}(BG)$ is the space of $W$-invariant quadratic forms on $\gt$.
The restriction of the Killing form is one such form, and any other is a constant multiple of this.
As a consequence, we see that $H^{1}(G)=H^{2}(G)=0$ and $H^{3}(G)=\R$.
This proves \autoref{H3symspace} in the case that $K$ is trivial.

Assume now that $K$ is a non-trivial closed, connected subgroup of $G$.
We shall use the spectral sequence of the following fibration:
$$
	\begin{matrix}
		G/K & \to & BK \\
		&& \downarrow\\
		&& BG.
	\end{matrix}
$$
That is:
\begin{equation}
	\label{specseq-symspace}
	H^{r}(BG) \otimes H^{s}(G/K)
	\Rightarrow
	H^{r+s}(BK).
\end{equation}
Let $S$ be a maximal torus in $K$ and $T\supset S$ be a maximal torus in $G$,
and let $W_{G}$ and $W_{K}$ be the corresponding Weyl groups.
From the spectral sequence in \autoref{specseq-symspace} we have an edge map
\[
	H^{\bullet}(BG) \to H^{\bullet}(BK).
\]
By \autoref{prop borel 1}, we may interpret this as a map
\[
	\R[\gt]^{W_{G}} \to \R[\gs]^{W_{K}},
\]
where $\gs$ is the Lie algebra of $S$.
This map has been determined by Borel:

\begin{proposition}[Proposition 28.2 of \cite{borel-thesis}]
	The above edge map is given by restricting a polynomial function on $\gt$ to
	the subspace $\gs$.
\end{proposition}

We can now finish proving our proposition.
Suppose that $G$ is simple and $K$ is non-trivial.
Then $H^{4}(BG)$ is generated by the Killing form.
Since the Killing form is negative definite, its restriction to $\gs$ is non-zero.
This shows that the edge map $H^{4}(BG) \to H^{4}(BK)$ is injective.
The $E_{2}$-sheet of the spectral sequence in \autoref{specseq-symspace}
looks like this:
\[
	\begin{tikzcd}
		H^{3}(G/K) & 0 & 0 & 0\\
		H^{2}(G/K) & 0 & 0 & 0\\
		H^{1}(G/K) & 0 & 0 & 0\\
		\R & 0 & 0 & 0 & H^{4}(BG)
	\end{tikzcd}
\]
These groups all remain unchanged until the $E_4$
sheet, where we have a map $H^{3}(G/K) \to H^4(BG)$.
From this we see that
\[
	H^{1}(G/K)
	=
	H^{1}(BK)
	=
	0
	\quad \text{ because $1$ is odd,}
\]
\[
	H^{2}(G/K)
	=
	H^{2}(BK)
	=
	H^{1}(K)
	=
	\gz(\gk)^{*}.
\]
Furthermore there is an exact sequence:
\[
	0 \to H^{3}(BK) \to H^{3}(G/K) \to H^{4}(BG) \to H^{4}(BK).
\]
We've seen that $H^{3}(BK)=0$ (because $3$ is odd) and the edge map $H^{4}(BG) \to H^{4}(BK)$ is injective,
so it follows that $H^{3}(G/K)=0$.
\end{proof}

We finally translate the result above into a more usable form.
In the following corollary, $G$ is a semi-simple, simply connected algebraic group over $\R$ and $K_\infty$ is a maximal compact subgroup of $G(\R)$.
We shall write $b_\R$ and $b_\C$ for the number of simple components of $G$ of real and of complex type.
We write $\gg$ and $\gk$ for the Lie algebras of $G(\R)$ and $K_\infty$ respectively, and we write
$\gz(\gk)$ for the centre of $\gk$.

\begin{corollary}
	\label{corollary sym space}
	With the notation introduced above, we have
	\[
		H^2(\gg,\gk,\R) = \gz(\gk)^*,
		\qquad
		\dim H^3(\gg,\gk,\R) = b_\C.
	\]
\end{corollary}

\begin{proof}
	Recall from section 1.6 of \cite{borel-wallach}, that the $\gg,\gk$-cohomology is the same as the singular cohomology of the
	compact symmetric space $X = H(\R)/K_\infty$,
	where $H$ is the compact form of $G$ over $\R$; in other words, $H(\R)$ is a maximal compact subgroup of $G(\C)$ containing $K_\infty$.
	
	Let $G= \prod_{i=1}^{b_\R+b_\C} G_i$,
	where each $G_i$ is a simple, simply connected group over $\R$; the groups $G_1,\ldots,G_{b_\R}$ are assumed to be of real type and the others are of complex type.
	The subgroup $K_\infty$ also decomposes as $\prod K_{i}$, where each $K_{i}$ is a maximal compact subgroup of $G_i(\R)$.
	Similarly the compact form $H$ decomposes
	as $\prod H_i$, where each $H_i$ is the compact form of $G_i$.
	This gives us a decomposition
	\[
		X = \prod X_i,
		\qquad \text{ where } X_i =H_i(\R) / K_i.
	\]
	Our aim is to calculate the cohomology of $X$
	using the K\"unneth formula.
	\begin{itemize}
		\item[Case 1.]
		In the case that $G_i$ is of real type,
		The group $G_i\times \C$ is simple over $\C$.
		Therefore $H_i(\R)$, being a maximal compact subgroup of $G_i(\C)$, is a simple, simply connected Lie group.
		By \autoref{H3symspace} we have
		\begin{align*}
			H^1(X_i) &=0,&
			H^2(X_i) &= \gz(\gk_i)^*,&
			H^3(X_i) &= 0.
		\end{align*}
		\item[Case 2.]
		Suppose instead that $G_i$ is of complex type.
		In this case $G_i \times \C$ splits as a direct sum of two simple groups and we have
		$G_i(\C) \cong G_i(\R) \times G_i(\R)$;
		the subgroups $G_i(\R)$ and $K_i$ are diagonally embedded in $G_i(\C)$.
		As $H_i(\R)$ is a maximal compact subgroup of $G_i(\C)$, we have $H_i(\R) = K_i \times K_i$.
		In this case, the compact symmetric space $X_i$
		is the quotient $(K_i \times K_i) / K_i$,
		which is homeomorphic to $K_i$.
		By \autoref{H3symspace} we have
		\begin{align*}
			H^1(X_i) &=0,&
			H^2(X_i) &= \gz(\gk_i)^*,&
			H^3(X_i) &= \R.
		\end{align*}
	\end{itemize}
	By the K\"unneth formula we have
	\[
		H^2(\gg,\gk,\R)
		=
		\bigoplus_i \gz(\gk_i)^*
		=\gz(\gk)^*,
	\]
	\[
		H^3(\gg,\gk,\R)
		=
		\bigoplus_{i \text{ of complex type}} \R
		=
		\R^{b_\C}.
	\]
\end{proof}

\bigskip

\sc{Department of Mathematics, University College London.}

\end{document}